\documentclass[12pt]{amsart}
\usepackage{amsmath,amsfonts,amsthm,amsopn,cite,mathrsfs,amssymb}

\usepackage{epsfig,verbatim}
\usepackage{subfigure,color,enumitem}

\numberwithin{equation}{section}

\newcommand{\round}[1]{\left\{#1\right\}}
\newcommand{\sech}{{\mathop{\rm sech}}}

\newcommand{\weakconv}{\xrightarrow{w}}
\newcommand{\WZ}{W_{\mathbb{Z}}}

\newcommand{\ctilg}{{\sum_{k \in \bZ} c_k g(\cdot-k)}}

\newcommand{\sett}[1]{\ensuremath{\left \{ #1 \right \}}}

\newcommand{\sisp}{V}

\newcommand{\tp}{totally positive}
\newcommand{\ltwo}{L^2(\bR )}

\newcommand{\ml}{m_\Lambda}
\newcommand{\mg}{m_\Gamma}
\newcommand{\mult}{m}

\newtheorem*{lemma*}{Lemma}

\DeclareMathOperator*{\supp}{supp}

\newcommand{\field}[1]{\mathbb{#1}}
\newcommand{\bR}{\field{R}}
\newcommand{\bN}{\field{N}}
\newcommand{\bZ}{\field{Z}}
\newcommand{\bC}{\field{C}}

\def\inv{^{-1}}

\newcommand{\abs}[1]{\lvert#1\rvert}
\newcommand{\norm}[1]{\lVert#1\rVert}
\newcommand{\bignorm}[1]{\big\lVert#1\big\rVert}

\newtheorem{tm}{Theorem}[section]
\newtheorem{lemma}[tm]{Lemma}
\newtheorem{prop}[tm]{Proposition}
\newtheorem{cor}[tm]{Corollary}

\newcommand{\Rst}{\bR}
\newcommand{\Zst}{\bZ}
\newcommand{\Nst}{\bN}

\begin{document}
\begin{abstract}
We study the problem of sampling with derivatives in shift-invariant spaces
generated by totally-positive functions of Gaussian type  or by  the
hyperbolic secant. We provide
sharp conditions in terms of weighted Beurling densities. As a
by-product we derive new results about multi-window  Gabor frames with
respect to vectors of Hermite functions or totally positive
functions.
\end{abstract}

\title[Sharp results on sampling with derivatives]
{Sharp results on sampling with derivatives in shift-invariant spaces
and multi-window Gabor frames}
\author{Karlheinz Gr\"ochenig}
\email{karlheinz.groechenig@univie.ac.at}

\author{Jos\'e Luis Romero}
\email{jose.luis.romero@univie.ac.at}
\address{Faculty of Mathematics \\
University of Vienna \\
Oskar-Morgenstern-Platz 1 \\
A-1090 Vienna, Austria}

\author{Joachim St\"ockler}
\email{joachim.stoeckler@math.tu-dortmund.de}

\address{Faculty of Mathematics \\
TU Dortmund \\
Vogelpothsweg 87 \\
D-44221 Dortmund, Germany }

\date{}
\thanks{J.~L.\ R.\
gratefully acknowledges support from the Austrian Science Fund (FWF):
P 29462 - N35.}
\maketitle

\section{Introduction and results}
In the problem of sampling with derivatives one tries to recover or
approximate a function by sampling a number of its derivatives. In
analogy to Hermite interpolation this procedure is sometimes called
Hermite sampling. For a well-defined problem one must fix a suitable
signal model, which in engineering is usually a space of bandlimited
functions (the Paley-Wiener space in mathematical terminology). In
recent years the more general model of shift-invariant spaces has
received considerable attention as a viable substitute for bandlimited
functions. See~\cite{AG00} for an early survey.

Hermite sampling can be seen as a purely mathematical problem in
approximation theory, but it is also informed by practical
considerations. Whereas a sample $f(\lambda )$ at a sampling point $\lambda
$ gives its  pointwise value, the derivative $f'(\lambda )$ measures
the trend of $f$ at $\lambda $, and higher derivatives  yield
information about the local
approximation by Taylor polynomials. In addition, by taking several
measurements at each point, one may hope to use fewer sampling
points.

Based on our experience, we will analyze Hermite sampling in
shift-invariant spaces that are generated by certain totally positive
functions. We will call
a function $g: \Rst \to \bC$  \emph{\tp\ of Gaussian type} if its Fourier
transform factors as
\begin{equation}\label{eq:tpaa}
  \hat g(\xi)= \prod_{j=1}^n (1+2\pi i\delta_j\xi)^{-1} \, e^{-c \xi
    ^2},\qquad \delta_1,\ldots,\delta_n\in\bR, c >0, n\in \bN \cup \{0\} \, .
\end{equation}
We study the problem of sampling with multiplicities in the  shift-invariant space
\begin{align*}
\sisp^p(g) = \big\{ f\in L^p(\Rst) : f = \sum _{k\in \bZ } c_k g(\cdot
  - k),  \, c\in
\ell ^p(\bZ )\big\},
\end{align*}
generated by a totally-positive function of Gaussian-type,
where $1\le p\le \infty$.
To describe the sampling process, we fix a sampling set $\Lambda
\subseteq \Rst$ and a multiplicity function  $\ml: \Lambda \to \bN$, and
call $(\Lambda,\ml)$ a  set with multiplicity. The number $m_\Lambda (\lambda )$
indicates how many  derivatives are sampled at  $\lambda \in \Lambda
$.

We then  say that $(\Lambda, \ml)$ is a sampling set for $\sisp^p(g)$
with $1\leq p <\infty$,  if
there exist constants $A,B>0$ such that
\begin{equation}\label{eq:lpstable}
A \norm{f}_p^p\leq\sum_{\lambda \in \Lambda} \sum_{j=0}^{\ml(\lambda)-1}
\abs{f^{(j)}(\lambda)}^p
\leq B \norm{f}_p^p, \qquad f \in \sisp^p(g) \, .
\end{equation}
If $p=\infty$,  a sampling set is defined by the inequalities
\begin{align} \label{eq:h7}
A \norm{f}_\infty \leq \sup_{\lambda \in \Lambda} \max_{0\le j\le\ml(\lambda)-1}
\abs{f^{(j)}(\lambda)}
\leq B \norm{f}_\infty, \qquad f \in \sisp^\infty(g) \, .
\end{align}

From a theoretical point of view the  sampling inequality
\eqref{eq:lpstable} completely solves the (Hermite) sampling
problem. We note that a sampling inequality always leads to a general
reconstruction algorithm based on frame theory~\cite{DS52}. In
addition, for localized generators
 the frame algorithm converges even in the correct $L^p$-norm~\cite{CG04}.  Thus
\eqref{eq:lpstable} is also  a first step towards the numerical
treatment of the sampling problem.

Our objective is the characterization of sampling sets satisfying the
sampling inequality~\eqref{eq:lpstable} and to obtain sharp conditions
on the sampling set. In Beurling's tradition of complex analysis we
will  characterize sampling sets in terms of a weighted version of
Beurling's lower density
\begin{align}
\label{eq_bdens}
D^{-}(\Lambda,\ml) := \liminf_{r \rightarrow \infty}
\inf_{x \in \Rst} \frac{1}{2r} \sum_{\lambda \in \Lambda \cap [x-r,x+r]} \ml(\lambda).
\end{align}

Within this setting  we can already formulate our main result.
\begin{tm}
\label{th_samp_der_tp}
Let $g$ be a totally positive function of Gaussian type. Let $\Lambda
\subseteq \Rst$
be a separated set and let $m_\Lambda:\Lambda \to \bN$ be a multiplicity
function  such that $\sup_{\lambda \in \Lambda} \ml(\lambda) <\infty$.

(i)  If
$D^{-}(\Lambda, m_\Lambda) >1$, then $(\Lambda, m_\Lambda)$ is a sampling set for
$\sisp^p(g)$ for every $1\le p\le \infty$.

(ii) Conversely, if $(\Lambda, m_\Lambda)$ is a sampling set for
$\sisp^2(g)$, then $D^{-}(\Lambda, m_\Lambda) \geq 1$.
\end{tm}
Theorem \ref{th_samp_der_tp} extends one of the results in \cite{grrost17} to sampling with
multiplicities. We also have an analogous density result for the shift-invariant space
generated by the hyperbolic secant.

\begin{tm}
\label{th_samp_der_sec}
Let $\psi(x)=\sech(a x)=\frac{2}{e^{a x}+e^{-a x}}$ be the hyperbolic secant. Let
$\Lambda \subseteq \Rst$
be a separated set and  $m_\Lambda $ be a multiplicity function  such
that $\sup_{\lambda \in \Lambda} \ml(\lambda) <\infty$.

(i) If
$D^{-}(\Lambda, m_\Lambda) >1$, then $(\Lambda, m_\Lambda)$ is a sampling set for
$\sisp^p(\psi)$ and every $1\le p\le \infty$.

(ii) Conversely, if $(\Lambda, m_\Lambda)$ is a sampling set for
$\sisp^2(\psi)$, then $D^{-}(\Lambda, m_\Lambda) \geq 1$.
\end{tm}

For comparison, we state the corresponding  sampling result for the Paley-Wiener space
$$
\mathrm{PW}^2 = \{ f\in \ltwo : \supp \, \hat f \subseteq [-1/2,1/2]\}
\, .
$$
The statement  is analogous to Theorems \ref{th_samp_der_tp}
and \ref{th_samp_der_sec} and is considered folklore among complex
analysts (we tested it!).
\begin{tm}
\label{th_samp_der_pw}
Let $\Lambda \subseteq \Rst$
be a separated set and let $m_\Lambda$ be a multiplicity function  such that
$\sup_{\lambda \in \Lambda} \ml(\lambda) <\infty$.

(i) If
$D^{-}(\Lambda, m_\Lambda) >1$, then $(\Lambda, m_\Lambda)$ is a sampling set for
$\mathrm{PW}^2$.

(ii)  Conversely, if $(\Lambda, m_\Lambda)$ is a sampling set for
$\mathrm{PW}^2$, then $D^{-}(\Lambda, m_\Lambda) \geq 1$.
\end{tm}
Although folklore, Theorem \ref{th_samp_der_pw}   does not seem to have been formulated
explicitly  in the literature.  A very
interesting  result involving divided differences of
samples was proved for the Bernstein space $\mathrm{PW}^\infty $ by Lyubarski
and Ortega-Cerda~\cite{LOC14}.
For the Fock space a   result similar to  Theorem \ref{th_samp_der_pw}
was derived early on  by Brekke and Seip \cite{bese93}.

 Theorems~\ref{th_samp_der_tp} and~\ref{th_samp_der_sec} have also several
consequences for Gabor systems.
Specifically, we characterize semi-regular sets $\Lambda \times
\beta\bZ$ that generate a multiwindow Gabor frame with respect to the
first $n$ Hermite functions or with respect to a specific  finite set of totally
positive functions. See Section~6 for the precise formulations.

In the literature most  sampling results for shift-invariant spaces
work with  the assumption that the sampling set $\Lambda $ is ``dense
enough''. However, when the sufficient density is made explicit, it is
usually  very far from the known
necessary density, even in dimension $1$. In fact,
until~\cite{grrost17} all authors use the  covering density or maximum
gap between samples,  and the density then depends  on some modulus of
continuity of the generator. See
~\cite{AF98} for one of the first nonuniform sampling theorems in
shift-invariant spaces,
\cite{Raz95} for nonuniform sampling with derivatives for bandlimited
functions, and~\cite{AGH17,SR16,selvan17} for more recent examples of
sufficient conditions for Hermite sampling in terms of the covering density.

In the light of \cite{grrost17} the sharp results for sampling with
derivatives are perhaps not surprising, but  they definitely go far
beyond the current state-of-the-art. Our main point is to show the
usefulness and power of the established methods, which consist of the
combination of Beurling's techniques, spectral invariance, complex
analysis, and the comparison of zero sets in different shift-invariant
spaces. We believe that these methods carry a high potential in many
other situations.

To arrive at sharp results, we combine several techniques.
Roughly, we proceed in three steps:

(i) We use Beurling's method of
weak limits and show that  the  sampling inequality
\eqref{eq:h7} for $p=\infty $ is equivalent to the fact that
every weak limit of integer translates  of $\Lambda $ is a uniqueness set for $V^\infty
(g)$. In this way we obtain a general characterization of sampling
sets \emph{without inequalities} (Theorem~\ref{th_wl_tp}).

(ii) To
switch between sampling inequalities for $p=\infty $ and $p<\infty $,
we use the theory of localized frames and Sj\"ostrand's beautiful
version of Wiener's Lemma for convolution-dominated
matrices~\cite{sj95}.
These two steps are part of a general mathematical formalism that can
be applied to many different situations. In particular, they work for
shift-invariant spaces with almost arbitrary generators.

(iii) The concrete understanding then rests on the analysis  of
uniqueness sets for a particular shift-invariant space $\sisp ^p(g)$, or in other
words, we need to analyze the zero sets of arbitrary functions in
$\sisp ^p(g)$. For instance, for the classical Paley-Wiener space this
is the
relation between  the
density of the zero set of an entire function and  its growth. This is
precisely the aspect where we develop new arguments.
Firstly, we observe that every function in $V^p(\phi )$ for  a
Gaussian generator $\phi $
possesses  an extension to an entire function, and secondly,
we can  relate the  real zeros of
some $f\in \sisp ^p(\phi )$ to  the  complex zeros of its analytic
  extension. A similar, but technically more involved  strategy works
  for the hyperbolic secant $\psi (x) = (e^{ax} - e^{-ax})\inv $.   In
a final step we relate the zero sets of functions in \emph{different}
shift-invariant spaces to each other. In this way we develop a direct
line of arguments and avoid the detour in\cite{grrost17}  via the characterization of
Gaussian Gabor frames.

The paper is organized as follows: Section~2 introduces the necessary
definitions for sampling in vector-valued shift-invariant
spaces. These provide a convenient language to formulate the problem of
sampling with derivatives. Section~3 then contains the main structural
characterization of sampling with derivatives and the necessary
density condition (Proposition~\ref{prop_nec}). Section~4 is devoted
to the investigation of the density of zero sets in shift-invariant
spaces. This is the part that contains the main arguments and new proof
ideas. The proofs of Theorems \ref{th_samp_der_tp}, \ref{th_samp_der_sec}, and \ref{th_samp_der_pw}
are then in Section \ref{sec_proofs}. In Section~6 we draw some
consequences of the sampling theorems with derivatives for
multi-window  Gabor frames. Finally Section 7 contains some of the
postponed proofs of the structural results in Sections~2 and ~3. As
these are essentially known, we explain only the necessary
modifications.

\section{Vector-valued shift-invariant spaces and sampling}
\subsection{Vector-valued shift-invariant spaces}
The treatment of sampling with derivatives requires us to formulate
several standard concepts for vector-valued functions. In this
section, we collect the precise definitions.
For the proper formulation of sampling results we make use of the Wiener amalgam space $W_0=W_0(\Rst)$, which consists of
continuous functions $g$ such that
\begin{align*}
     \|g\|_{W} := \sum_{k\in\bZ} \max_{x\in [k,k+1]} |g(x)| <\infty.
\end{align*}

Let $G=(G^1, \ldots, G^N) \in (W_0(\Rst))^N$. We consider the
vector-valued shift-invariant space
\begin{align}
\sisp^p(G) := \left\{ \sum_{k \in \Zst} c_k G(\cdot-k): c \in \ell^p(\Zst) \right\}
\end{align}
as a subspace of $(L^p(\Rst))^N$ with norm
\begin{align*}
\norm{(F^1, \ldots, F^N)}_p := \left(\sum_{j=1}^N \norm{F^j}^p_p\right)^{1/p}, \qquad 1 \leq
p < \infty,
\end{align*}
and $\norm{(F^1, \ldots, F^N)}_\infty = \max_{j=1,\ldots,N} \norm{F^j}_\infty$.
We always assume that $G$ has stable integer shifts, i.e.
\begin{align}
\label{eq_vec_riesz}
\bignorm{\sum_{k \in \Zst} c_k G(\cdot-k)}_p \asymp \norm{c}_p.
\end{align}

\subsection{Sampling and weak limits}
We consider tuples of sets $\vec{\Lambda}=(\Lambda^1, \ldots, \Lambda^N)$ with $\Lambda^j
\subseteq \Rst$.
We say that $\vec{\Lambda}$ is a \emph{sampling set for} $\sisp^p(G)$,
$1\leq p \leq \infty $, if
\begin{align}
\label{eq_vec_samp}
\norm{F}_p \asymp \left(\sum_{j=1}^N
  \norm{F^j|\Lambda^j}^p_p\right)^{1/p} = \Big( \sum _{j=1}^N \sum
  _{\lambda \in \Lambda ^j} |F^j(\lambda )|^p\Big)^{1/p} , \qquad  \text{ for all }  F \in \sisp^p(G).
\end{align}
For $p=\infty$ the condition reads as $\norm{F}_\infty \asymp \max_{j=1,\ldots,N}
\norm{F^j|\Lambda ^j}_\infty$. We say that $\vec{\Lambda}$ is a \emph{uniqueness set for}
$\sisp^p(G)$ if whenever $F \in \sisp^p(G)$ is such that $F^j \equiv 0$ on $\Lambda^j$,
for all $j=1, \ldots, N$, then $F \equiv 0$. Clearly, sampling sets are also uniqueness sets.
\medskip

We first recall Beurling's notion of a weak limit of a sequence of sets.
A sequence $\{\Lambda_n: n \geq 1\}$ of subsets of $\bR$ is said to \emph{converge weakly}
to a set $\Lambda \subseteq \bR$,
denoted $\Lambda_n \weakconv \Lambda$, if for every open bounded interval $(a,b)$ and
every $\varepsilon >0$, there exist $n_0 \in \bN$ such that for all $n \geq n_0$
\begin{align*}
\Lambda_n \cap (a,b) \subseteq \Lambda + (-\varepsilon,\varepsilon)
\mbox { and }
\Lambda \cap (a,b) \subseteq \Lambda_n + (-\varepsilon,\varepsilon).
\end{align*}

We let $\WZ(\Lambda)$ denote the class of all sets $\Gamma$ that can be obtained as weak
limits of integer translates of $\Lambda$, i.e., $\Gamma \in \WZ(\Lambda)$ if
there exists a sequence $\{k_n: n \geq 1\} \subseteq \mathbb{Z}$
such that $\Lambda + k_n \weakconv \Gamma$. We extend this notion to tuples of sets as
follows.
Given two $N$-tuples of sets $\vec\Lambda=(\Lambda^1,\ldots,\Lambda^N)$
and $\vec\Gamma=(\Gamma^1, \ldots, \Gamma^N)$, we say that
$\vec\Gamma \in \WZ(\vec\Lambda)$ if there exists a sequence $\{k_n: n \geq 1\} \subseteq
\mathbb{Z}$
such that $\Lambda^j + k_n \weakconv \Gamma^j$ for all $1 \leq j \leq N$.
(Note that the limits involve \emph{the same sequence} $\{k_n: n \geq 1\}$ for all $j$.)

The following is a vector-valued extension of \cite[Theorem 3.1]{grrost17}.
\begin{tm}
\label{th_wl_vec}
Let $G=(G^1, \ldots, G^N) \in (W_0(\Rst))^N$ have stable integer shifts and
let $\vec{\Lambda}=(\Lambda^1, \ldots, \Lambda^N)$ be a tuple of separated sets.
Then the following are equivalent.
\begin{itemize}
\item[(a)] $\vec\Lambda$ is a sampling set for $\sisp^p(G)$ for some $p \in
[1,\infty]$.
\item[(b)] $\vec\Lambda$ is a sampling set for $\sisp^p(G)$  for all $p \in
[1,\infty]$.
\item[(c)] Every weak limit $\vec\Gamma \in \WZ(\vec\Lambda)$ is a sampling
set
for $\sisp^\infty(G)$.
\item[(d)] Every weak limit $\vec\Gamma \in \WZ(\vec\Lambda)$ is a set of
uniqueness for $\sisp^\infty(G)$.
\end{itemize}
\end{tm}
The proof is similar to the scalar-valued version; a sketch of the
proof is given in Section \ref{sec_post}.

\section{Sampling with multiplicities}
\subsection{Sets with multiplicities and derivatives}
For $N \in \Nst$ we let $W^N_0=W^N_0(\Rst)$ be the class of functions $g$ having
derivatives up to order $N-1$ in $W_0(\Rst)$.
 For a set with multiplicity $(\Lambda,\ml)$, we define its height
as $\sup_\lambda m_\Lambda(\lambda)$. When sampling in shift-invariant spaces with generators on $W^N_0(\Rst)$ we assume
that the sampling sets have height $\leq N$. The lower density of $(\Lambda,\ml)$ is defined by
\eqref{eq_bdens}.

\subsection{Sampling with derivatives}
We now describe how the problem of sampling with multiplicities can be reformulated in terms of sampling of
vector-valued functions.

Let a generator $g \in W^N_0(\Rst)$ with stable integer
shifts be given. We define $G \in \left(W_0(\Rst)\right)^N$ by choosing as components the
derivatives of $g$, so
\[G=(g,g^{(1)}, \ldots, g^{(N-1)}).\]
There is an obvious one-to-one correspondence between
$f=\sum_k c_k g(\cdot-k)\in \sisp^p(g)$ and
$F=(f,f^{(1)}, \ldots, f^{(N-1)})\in \sisp^p(G)$. In addition, since $g$ has stable integer shifts, we have the norm equivalence
$$   \|f\|_p \asymp \|c\|_p .$$
Furthermore, since $g^{(j)}\in W_0(\Rst)$ for $1\le j\le N-1$, there is a constant $ B>0$  such that
$$ \|f^{(j)}\|_p \le B\|c\|_p\quad \mbox{for}\quad 1\le j\le N-1,$$
and this implies
$$   \|f\|_p \asymp \|c\|_p \asymp \|F\|_p.$$
This shows that $G$ has stable
integer shifts in the sense of \eqref{eq_vec_riesz}.

Second, given a set with multiplicity $(\Lambda,\ml)$
and height at most $N <\infty$,
we consider the tuple sets
$\vec\Lambda=(\Lambda^1, \ldots, \Lambda^N)$ given by
\begin{align*}
\Lambda^k := \left\{\lambda \in \Lambda: \ml(\lambda) \geq k \right\}.
\end{align*}
Note that $\Lambda^1=\Lambda$. The connection between vector-valued
sampling and sampling with derivatives is stated  in the following
lemma, which is a direct consequence of our notation.
\begin{lemma}
A set with multiplicity $(\Lambda,\ml)$ and height at most $N<\infty$
is a sampling set
for $\sisp^p(g)$ in the sense of \eqref{eq:lpstable},  if and only if
$\vec\Lambda=(\Lambda^1,\ldots,\Lambda^N)$ is a sampling set for $\sisp^p(G)$, with $G=(g,g^{(1)}, \ldots, g^{(N-1)})$.
\end{lemma}

Finally, we interpret a weak limit $\vec\Gamma \in \WZ(\vec\Lambda)$ as a set with
multiplicity by setting
$\Gamma := \Gamma^1$ and
\begin{align*}
\mg(\gamma) := \max\{j \in \Nst: \gamma \in \Gamma^j\}, \qquad  \gamma \in \Gamma.
\end{align*}
In order to keep our notations consistent, we also write  $(\Gamma,\mg)\in \WZ(\Lambda,\ml)$  for the current situation.

For \emph{separated sets} $\Lambda$,  i.e., $\inf\{\abs{\lambda-\lambda'}: \lambda, \lambda'
\in \Lambda, \lambda \not= \lambda'\}>0$,  we have the following alternative description of
weak convergence.
\begin{prop}
\label{prop_wstar}
Let $(\Lambda, \ml)$ be a separated set with multiplicity and finite height $N$,
let $(\Gamma, \mg)$ be a set with multiplicity, and $\{k_n: n \geq 1\} \subseteq \Zst$. Then
$\Lambda^j-k_n \weakconv \Gamma^j$, as $n \longrightarrow \infty$ for all $j=1,\ldots,N$
if and only if
\begin{align*}
\sum_{\lambda \in \Lambda} \ml(\lambda) \delta_{\lambda-k_n} \longrightarrow
\sum_{\gamma \in \Gamma} \mg(\gamma) \delta_\gamma,
\mbox{ as }n \longrightarrow \infty,
\end{align*}
in the $\sigma(C^*_c,C_c)$ topology (where $C_c$ denotes the class of continuous
functions with compact support).
\end{prop}
A proof of Proposition \ref{prop_wstar} is given in Section \ref{sec_post}.
As a consequence, we obtain the following lemma; see, e.g.
\cite[Lemma 7.1]{grrost17} for a proof without multiplicities.
\begin{lemma}
\label{lemma_sep}
Let $(\Lambda, \ml)$ be a separated set with multiplicity and finite height,
and let $(\Gamma, \mg) \in \WZ(\Lambda, \ml)$. Then
$D^{-}(\Gamma, \mg) \geq D^{-}(\Lambda, \ml)$.
\end{lemma}

\subsection{Characterization of sampling with derivatives}
Theorem \ref{th_wl_vec} can be recast in terms of sampling with derivatives.
\begin{tm}
\label{th_wl_tp}
Let $g \in W^N_0(\Rst)$ have stable integer shifts and
let $(\Lambda, m_\Lambda)$ be a separated set with multiplicity and height at most $N < \infty$.
Then the following are equivalent.
\begin{itemize}
\item[(a)] $(\Lambda, m_\Lambda)$ is a sampling set for $\sisp^p(g)$ for some $p \in
[1,\infty]$.
\item[(b)] $(\Lambda, m_\Lambda)$ is a sampling set for $\sisp^p(g)$ for all $p \in
[1,\infty]$.
\item[(c)] Every weak limit $(\Gamma, \mg) \in \WZ(\Lambda, m_\Lambda)$ is a sampling
set
for $\sisp^\infty(g)$.
\item[(d)] Every weak limit $(\Gamma, \mg) \in \WZ(\Lambda, m_\Lambda)$ is a set of
uniqueness for $\sisp^\infty(g)$.
\end{itemize}
\end{tm}

For bandlimited functions, only some of the implications in Theorem \ref{th_wl_tp} are valid.
These are formulated in terms of the Bernstein space
$\mathrm{PW}^\infty$ of continuous bounded functions which are Fourier transforms
of distributions supported on $[-1/2,1/2]$.
\begin{tm}
\label{th_wl_pw}
Let $(\Lambda, m_\Lambda)$ be a separated set with multiplicity and finite height.
Then the following are equivalent.
\begin{itemize}
\item[(a)] $(\Lambda, m_\Lambda)$ is a sampling set for $\mathrm{PW}^\infty$.
\item[(c)] Every weak limit $(\Gamma, \mg) \in \WZ(\Lambda, m_\Lambda)$ is a sampling
set
for $\mathrm{PW}^\infty$.
\item[(d)] Every weak limit $(\Gamma, \mg) \in \WZ(\Lambda, m_\Lambda)$ is a set of
uniqueness for $\mathrm{PW}^\infty$.
\end{itemize}
\end{tm}
As a replacement for the $L^2$ part of Theorem \ref{th_wl_tp}, we have the following result.
\begin{prop}
\label{prop_pw}
Let $(\Lambda, m_\Lambda)$ be a separated set with multiplicity and finite height,
and assume that $(\Lambda, m_\Lambda)$ is a sampling set for $\mathrm{PW}^\infty$. Then, for
every $\alpha \in (0,1)$, $(\alpha \Lambda, m_\Lambda)$ is a sampling set for $\mathrm{PW}^2$.
\end{prop}
Theorem \ref{th_wl_pw} and Proposition \ref{prop_pw} are due to Beurling \cite{be66, be89}
(without multiplicities) - see also \cite[Theorem 2.1]{olul12}. A slight modification of
the arguments yields the case with multiplicities.

\goodbreak
\subsection{Necessary density conditions}
\begin{prop}
\label{prop_nec}
Let $g \in W^N_0(\Rst)$ have stable integer shifts and
let $(\Lambda, m_\Lambda)$ be a separated set with multiplicity and height at most $N < \infty$.
If $(\Lambda, m_\Lambda)$ is a sampling set for $V^2(g)$, then
$D^{-}(\Lambda, m_\Lambda) \geq 1$.

A similar statement holds for the Paley-Wiener space $\mathrm{PW}^2$.
\end{prop}
Proposition \ref{prop_nec} follows from standard results on density of frames,
see e.g. \cite{bchl06, fghkr}. See Section \ref{sec_post} for a sketch of a proof.

\section{Density of zero sets in shift-invariant spaces}
We derive  sharp upper bounds for the density of real zeros of
functions in shift-invariant spaces with special generators. First,
 we  use methods of complex analysis when the generator is a Gaussian
(Section 4.1) or a hyperbolic secant (Section 4.2).  The
results and arguments are  similar for both cases, but the case of the
hyperbolic secant requires considerably more work and the analysis of
meromorphic functions.  In Sections 4.3 and 4.4  we then analyze the
zero sets in shift-invariant spaces generated by a totally positive
function of Gaussian type by means of a comparison theorem.

\subsection{The Gaussian}
We now consider Gaussian functions $\phi_a(x) := e^{-a x^2}$  with $a>0$.
\begin{lemma} \label{l1}
  Every $f=\sum_k c_k \phi_a(\cdot-k)\in \sisp ^\infty (\phi_a )$ possesses an extension to an
  entire function satisfying the growth estimate
\begin{align}
\label{eq_growth}
|f(x+iy)| \leq C \|c\|_\infty e^{a y^2}    \qquad x,y \in \bR \, .
\end{align}
\end{lemma}
\begin{proof}
 Using
$$
e^{-a (x+iy-k)^2} = e^{a y^2} e^{-2a i xy} e^{2a i ky} e^{-a
  (x-k)^2} \, ,
$$
we obtain that
\begin{equation}
  \label{eq:1}
  f(x+iy) = e^{a y^2} e^{-2a i xy} \sum _{k\in \bZ } c_k e^{2a i ky} e^{-a
  (x-k)^2} \, .
\end{equation}
Consequently
$$
|f(x+iy)| \leq e^{a y^2} \|c\|_\infty \sum_{k\in \bZ } e^{-a (x-k)^2} \, ,
$$
and we may take $C= \sup _{0\leq x\leq 1} \sum_{k\in \bZ } e^{-a
  (x-k)^2}$. Clearly,  $x+iy \mapsto f(x+iy)$ is an entire function.
\end{proof}
Our key observation relates the real zeros of $f\in
\sisp^\infty(\phi_a)$ to the zeros of its analytic extension.

\begin{lemma} \label{l2}
Let $f\in \sisp ^\infty (\phi _a)$ and $\lambda \in \bR $ be a zero of
$f$ with multiplicity $m$. Then for every $l\in \frac{\pi}{a}\bZ $, $\lambda + il$
is a zero of the analytic extension of $f$ with the same multiplicity
$m$.  In particular, if $f^{(j)}(\lambda ) = 0$ for $j= 0, \dots , m-1$, then
  $f^{(j)}(\lambda +il) = 0$ for $j= 0, \dots , m-1$ and  all $l\in \frac{\pi}{a}\bZ
    $.
  \end{lemma}

  \begin{proof}
By \eqref{eq:1} we obtain that
\begin{align*}
  f(\lambda +il) = e^{a l^2} e^{-2a i \lambda l} \sum _{k\in \bZ } c_k e^{2a i kl}
e^{-a
  (\lambda -k)^2}
  = e^{a l^2} e^{-2a i \lambda l} f(\lambda)
  = 0,
\end{align*}
because $e^{2a i kl}=1$ for all $l\in \frac{\pi}{a}\bZ$.

For higher multiplicities we argue as follows.
Note first that
$\frac{d^j}{dx ^j}(e^{-a x^2}) = p_j(x) e^{-a x^2}$ for a polynomial of degree
$j$  satisfying the recurrence relation $p_{j+1} (x) = -2a x p_j(x)
+ p_j ' (x)$.  It follows that the set $\{ p_j: j=0, \dots ,
m-1\}$ is a basis for the polynomials of degree smaller than $m$.

Now assume that $f\in \sisp ^\infty (\phi_a )$ and  $f^{(j)}(\lambda ) =
0$ for $j=0, \dots , m-1$. Then
$$
\sum _{k\in \bZ } c_k p_j(\lambda-k) e^{-a (\lambda -k)^2} = 0
\qquad \text{ for } j= 0, \dots , m-1 \, .
$$
This implies that for every polynomial $q$ of degree $<m$
\begin{equation}
  \label{eq:3}
\sum _{k\in \bZ } c_k q(\lambda-k) e^{-a (\lambda -k)^2} = 0  \, .
\end{equation}
We now proceed as in~\eqref{eq:1} and find that, for
$j=0,\ldots,m-1$,
\begin{align*}
f^{(j)}(\lambda +il) &= \sum _{k\in \bZ } c_k p_j(\lambda-k+il )
                       e^{-a (\lambda -k+il )^2} \\
&= e^{a l^2} e^{-2a i \lambda l} \sum _{k\in \bZ } c_k
  p_j(\lambda-k+il ) e^{2a i kl } e^{-a (\lambda -k)^2} \, .
  \end{align*}
Note that $e^{2a i kl}=1$ for all $l\in \frac{\pi}{a}\bZ$.
We insert  the Taylor expansion of $p_j$ at $\lambda -k$, i.e.,
\[p_j(\lambda -k +
il) = \sum _{r=0}^j p_j^{(r)}(\lambda -k) \frac{(il)^r}{r!},\]
and we
obtain that
$$
f^{(j)}(\lambda +il)  = e^{a l^2} e^{-2a i \lambda l} \sum
_{r=0}^j \frac{(il)^r}{r!} \sum _{k\in \bZ } c_k
  p_j^{(r)}(\lambda -k)  e^{-a (\lambda -k)^2} \, .
$$
Since each $p_j^{(r)}$ is a polynomial of degree $< m$, \eqref{eq:3}
implies that $f^{(j)}(\lambda +il) = 0$ for all $l\in \frac{\pi}{a}\bZ $ and $j=0,
\dots , m-1$. This shows that the multiplicity of $\lambda + i l$
is at least that of $\lambda$. Reversing the roles of $\lambda$ and $\lambda + i l$ we see that
the multiplicities are actually equal.
\end{proof}

We recall Jensen's formula, which relates the
number of zeros $n(r)$ in a disk $B(0,r)$
to the growth of an entire function by the identity
\begin{align} \label{jens}
\int_0^R \frac{n(r)}{r} dr = \frac{1}{2\pi} \int_0^{2\pi} \log\abs{f(R
  e^{i\theta})} d\theta - \log\abs{f(0)} \, .
\end{align}
 This is our main tool (from complex analysis) to prove the following
 result about the density of real zeros of functions in a
 shift-invariant space.

\begin{tm} \label{tm1}
Let $\phi_a(x) = e^{-a x^2}$  with $a>0$.
Let $f\in \sisp ^\infty (\phi_a) \setminus \{0\}$ and $N_f$ its set of real zeros, with
multiplicities $\mult_f(x)$, $x \in N_f$. Then $D^-(N_f, \mult_f) \leq 1$.
\end{tm}

\begin{proof}
Note that $N_f = \{\lambda \in \bR: f(\lambda ) =0\} $ is the set of
\emph{real} zeros of $f$. By Lemma \ref{l2}, the  set of
\emph{complex} zeros of (the analytic extension of)
$f$ contains the set $N_f + i\frac{\pi}{a}\bZ \subseteq \bC $, and, moreover, multiplicities are preserved.

To prove the theorem, we argue indirectly and   assume that $D^-(N_f,\mult_f) >1$.
Then there exists $\nu >1$ and $R_0$, such that
$$
\sum_{\lambda\in N_{f} \cap [x,x+r]} m_f(\lambda) \geq \nu r \qquad \text{ for all } x\in \bR, \, r
\geq R_0 \, .
$$

Let $n(r)$ be the number of complex  zeros of $f$ inside the open disk
$B(0,r)\subseteq \bC $  counted with
multiplicities.
Let us assume for the moment that $f(0)\not=0$.

The right-hand side of Jensen's formula~\eqref{jens} can be estimated,
by means of the growth estimate~\eqref{eq_growth}, as
 \begin{equation}
   \label{eq:5}
\frac{1}{2\pi} \int_0^{2\pi} \log\abs{{f}(R
  e^{i\theta})} \, d\theta
  - \log\abs{{f}(0)}
  \leq  A+\frac{1}{2\pi} \int_0^{2\pi}
a R^2 \sin ^2 \theta \,  d\theta =  A+\frac{a R^2}{2} \, ,
 \end{equation}
 where $A:= -\log{\abs{f(0)}}+ \log(\|c\|_\infty C)$.

 To estimate the left-hand side of \eqref{jens},
 we choose $R\in \bN
$ and  $R\geq
 R_0$ and partition $[-R^2, R^2) = \bigcup _{k=-R}^{R-1} [kR,
 (k+1)R)$. On each interval there are at least $\nu R$ real zeros of ${f}$
counted with multiplicity. By symmetry it is enough to consider intervals $[kR,(k+1)R)$ with $0\le k\le R-1$. By Lemma \ref{l2},
for each real zero $\lambda \in [kR, (k+1)R)$, with a certain multiplicity $m$, there are $2
\lfloor \frac{a}{\pi}\sqrt{(R^2)^2 - \lambda ^2} \rfloor +1
\geq  \frac{2a}{\pi}\sqrt{ R^4 - (k+1)^2R^2}-1$ complex zeros $\lambda
+il, l \in \tfrac{\pi}{a} \bZ$ in the disk $B(0,R^2)$,
each with
multiplicity $m$. By counting with multiplicities, there are at least
$$
\nu R \left( \frac{2a}{\pi}\sqrt{ R^4 - (k+1)^2R^2}-1\right)$$
complex zeros in $B(0,R^2)$ with real part in $[kR, (k+1)R)$ where $0 \le k\le R-1$. By summing over (positive and negative) $k$,
we obtain the following lower bound for the number
of complex zeros of ${f}$  in $B(0,R^2)$:
$$
n(R^2) \geq
2 \nu R \sum _{k=0}^{R-1}\left( \frac{2a}{\pi}\sqrt{ R^4 - (k+1)^2R^2}-1\right) =
 \frac{4 \nu aR^4}{\pi}\sum _{k=0}^{R-1} \frac{1}{R}\sqrt{1 - \frac{(k+1)^2}{ R^2}}
  - 2\nu R^2\, .
$$
The last sum is a Riemann sum of the integral $\int _{0}^1 \sqrt{1-x^2} \,
dx = \pi /4$. Let $\epsilon >0$ and $R_1\ge R_0$ satisfy $\beta:=\nu (1 -
\epsilon-\frac{2}{aR_1^2} ) > 1 $. Then, for some $R_2\ge R_1$ and all $R \geq R_2$,
we conclude that
$$
n(R^2) \geq a \nu R^4 \left(1 - \epsilon -\frac{2}{aR^2}\right)\ge a\beta R^4,
$$
or, equivalently,
$$
n(r) \geq a\beta r^2  \mbox{ for }
r \geq R_2^2.
$$
Therefore, the left-hand side of \eqref{jens} can be estimated as
$$
\int _{0} ^R \frac{n(r)}{r} \, dr  \geq \int _{R_2^2} ^R \frac{n(r)}{r} \, dr
\geq a\beta \left(\frac{R^2}{2} - \frac{R_2^4}{2} \right) \, .
$$
Since $\beta > 1 $, this estimate is
incompatible with the growth of ${f}$ as encoded in
\eqref{eq:5}. Therefore $D^-(N_{{f}},m_f) > 1$ is impossible.
This concludes the proof for $f$ such that $f(0) \not= 0$.

If $f(0)=0$, we let $n \geq 0$ be the vanishing order of $f$ at $0$ and
apply the previous argument to $\tilde{f}(z):=z^{-n}f(z)$. Alternatively,
one can verify directly that if $f \not\equiv 0$, then
there exists $k \in \bZ$ such that $f(k) \not=0$, and consider
$\tilde{f}(x)=f(x+k)$ with $\tilde f \in V^\infty (\phi )$.
\end{proof}

\subsection{The hyperbolic secant}
Let $\psi _a(x)=\sech(a x)=\frac{2}{e^{a x}+e^{-a x}}$. Our goal is to
study the  shift-invariant space generated by $\psi _a$.
While in \cite{grrost17} we studied $\sisp^2(\psi _a)$ by exploiting a connection
to Gabor analysis, and a certain representation of the Zak transform of $\psi _a$
due to Janssen and Strohmer \cite{jast02}, here we consider meromorphic extensions
of the functions in $\sisp^\infty(\psi _a)$.

We introduce the following notation.
For real $x$ we denote the roundoff error to the nearest integer as
 $\round{x}:= x-l$, where $l\in\bZ$ and $\round{x}\in[-1/2,1/2)$.

\begin{lemma}\label{l3}
 Every $f=\sum_{k\in\bZ} c_k \psi _a(\cdot-k) \in V^\infty(\psi _a)$ has an extension to a
meromorphic function
on $\bC$ with poles in
$$ P_f \subseteq P := \bZ+\frac{i\pi}{a}\left(\frac{1}{2}+\bZ\right).$$
Moreover, every pole of $f$ is simple and $f$ satisfies the growth estimate
\begin{equation}\label{eq:boundf}
\begin{array}{rcl}
   \abs{f(x+iy)} &\le&   C \|c\|_\infty \abs{\psi _a(\round{x}+iy)} \\[5pt]
&\le& C \|c\|_\infty \min\{\abs{a\round{x}}^{-1}, \abs{2\round{\tfrac{ay}{\pi}-\tfrac{1}{2}}}^{-1}\}.
\end{array}
\end{equation}
\end{lemma}

\begin{proof}  The meromorphic function $\psi _a(z)=\sech(a z)$ has simple poles on the
imaginary axis
at $\frac{i\pi}{a}\left(\frac{1}{2}+\bZ\right)$.
The identity
\begin{eqnarray*}
      \abs{\cosh a(x-k+iy)} &=& (\sinh^2 a(x-k) + \cos^2 a y)^{1/2}
\end{eqnarray*}
shows that $\abs{\psi _a (x-k+iy)} \lesssim e^{-a\abs{x-k}}$, if
$\abs{x-k} \geq 1$ and $y$ is arbitrary.
We consider the covering  of $\bC$ given by
\begin{align*}
U_{s,t} := \left\{x+iy \in \bC: \abs{x-s}< 3/4, \abs{y-\frac{\pi}{a}(t+1/2)} < \frac{3\pi}{4a} \right\},
\qquad s,t \in \bZ.
\end{align*}
On $U_{s,t}$, the partial sums
\begin{align*}
f_N(x+iy)=\sum_{k: \abs{k-s} \leq N } c_k \psi _a(x-k+iy)
\end{align*}
have at most a simple pole at $s + \tfrac{i\pi}{a}(\tfrac{1}{2} + t)$ and are otherwise analytic.
Since, for $x+iy \in U_{s,t}$,
\begin{align*}
\sum_{k: \abs{k-s} > N } \abs{c_k} \abs{\psi _a(x-k+iy)}
\lesssim \norm{c}_\infty \sum_{k: \abs{k-s} > N } e^{-a\abs{x-k}}
\lesssim \norm{c}_\infty e^{-N} \, ,
\end{align*}
the partial sums $f_N$ converge uniformly on
$U_{s,t} \setminus \left\{s + \tfrac{i\pi}{a}(\tfrac{1}{2} + t)\right\}$
to an analytic extension of $f$. More precisely,
\begin{align*}
\sup_{z \in U_{s,t} \setminus \left\{s + \tfrac{i\pi}{a}(1/2 + t)\right\}}
\abs{f(z)-f_N(z)} \longrightarrow 0, \qquad \mbox{as } N \longrightarrow \infty.
\end{align*}
(Note that this is stronger than the usual uniform convergence on
compact sets.)
This fact implies that $f$ has at most a simple pole at $z=s +
\frac{i \pi}{a}(1/2 + t)$. Hence,
$f$ is meromorphic on $\bC $ with at most simple poles in
$\bZ+\frac{i\pi}{a}\left(\frac{1}{2}+\bZ\right)$.

For the growth estimate \eqref{eq:boundf} we let
$x+iy \in \mathbb{C} \setminus P_f$ and write
 $x=l+\round{x}$ with
$l\in\bZ$ and $\round{x}\in [-1/2,1/2)$. Then we have
$$ \abs{f(x+iy)} \le \abs{\psi _a (\round{x}+iy)}\left(\abs{c_l}+
\sum_{k\ne l} \Big| \frac{c_k\psi _a (x-k+iy)}{\psi _a (\round{x}+iy)}
\Big| \right).$$
For all $k\ne l$ we observe that
$\abs{x-k} \geq \abs{l-k} -  \abs{\round{x}}\ge \frac{1}{2} \geq
  \abs{\round{x}}$. Therefore, we
   have
 $\sinh^2 a (x-k)\ge \sinh^2 a \round{x}$.
Since  the rational function $r(y)=\tfrac{c+y}{d+y}$ with $d\ge c\ge
0$, $d>0$  is increasing for $y>0$, we obtain that, for all $k\ne l$,
$$  \frac{\abs{\psi _a (x-k+iy)}^2}{\abs{\psi _a (\round{x}+iy)}^2} =
 \frac{\sinh^2 a \round{x}+\cos^2 a y}{\sinh^2 a(x-k)+\cos^2 a y}\le \frac{\sinh^2 a
\round{x}+1}{\sinh^2 a(x-k)+1}
 =\frac{\cosh^2 a\round{x}}{\cosh^2 a(x-k)},$$
and furthermore
$$\frac{\cosh a\round{x}}{\cosh a(x-k)}= \frac{e^{a\abs{\round{x}}}(1+e^{-2a\abs{\round{x}}})}{e^{
a\abs{x-k}}(1+e^{-2a\abs{x-k}})}
\le 2e^a e^{-a\abs{k-l}}.$$
 Therefore,  we have
 $$ \abs{c_l}+
\sum_{k\ne l} \left|\frac{c_k\psi _a (x-k+iy)}{\psi _a (\round{x}+iy)}\right| \le \|c\|_\infty
\left(1+ 2e^a\sum_{k\ne l} e^{-a\abs{k-l}}  \right) \le C\|c\|_\infty.$$
This proves the first inequality in \eqref{eq:boundf}. For the second inequality note that
$$ \abs{\sinh a x}\ge \abs{a  x}\quad\mbox{for all}\quad x\in\bR$$
and, by periodicity and elementary trigonometric  identities,
$$ \abs{\cos a y}=  \abs{\sin \pi \round{\tfrac{ay}{\pi}-\tfrac{1}{2}}} \ge 2\abs{\round{\tfrac{ay}{\pi}-\tfrac{1}{2}}}\quad\mbox{for all}\quad y\in
\bR.$$
Hence, we obtain
$$\abs{\psi _a (\round{x}+iy)}= \left( \sinh^2 a\round{x} +\cos^2 a y\right)^{-1/2}
\le \min\{\abs{ a \round{x}}^{-1},\abs{2\round{\tfrac{ay}{\pi}-\tfrac{1}{2}}}^{-1}\},$$
which gives the second inequality in \eqref{eq:boundf}.
\end{proof}

The following result is an analogue of Lemma \ref{l2} for $\sisp^\infty(\psi _a)$.

\begin{lemma} \label{l4}
Let $f\in \sisp ^\infty (\psi _a )$ and $\lambda \in \bR $ be a zero of
$f$ with multiplicity $m$. Then for every $l\in \tfrac{\pi}{a}\bZ $, $\lambda + il$
is a zero of the meromorphic extension of $f$ with the same multiplicity
$m$.
  \end{lemma}

\begin{proof}
For every $x\in\bR$ and $l=\tfrac{\pi t}{a}\in\tfrac{\pi}{a}\bZ$ we have
$$ \cosh  a(x+il)=\cosh  a x\cos  a l +i \sinh  a x\sin  a l= (-1)^t \cosh  a x.$$
Therefore, every
$f=\sum_{k\in\bZ} c_k \psi _a(\cdot-k) \in \sisp^\infty(\psi _a)$ satisfies
$$ f(x+il)=\sum_k c_k (-1)^t \psi _a  (x-k)=(-1)^t f(x).$$
This implies that the Taylor expansions of $f$ around $z_0= x\in\bR$ and around
$z_l= x+il$  have exactly the same coefficients, up to a factor $(-1)^t$. In particular,
$f^{(j)}(\lambda)=0$ holds for some $\lambda\in\bR$ and $j\ge 0$ if and only if
$f^{(j)}(\lambda+il)=0$
for all $l\in\tfrac{\pi}{a}\bZ$.
\end{proof}

Let $n(r)$ denote the difference of
the number of zeros and the number of poles of
$f$
in the closed disk $\overline{B(0,r)}$, counted with multiplicities. Jensen's formula for
meromorphic functions $f$ with
$f(0)\not\in
\{0,\infty\}$ says that
\begin{align} \label{jens2}
\int_0^r \frac{n(t)}{t} dt = \frac{1}{2\pi} \int_0^{2\pi} \log\abs{f(r
  e^{i\theta})} d\theta - \log\abs{f(0)},
\end{align}
see e.g. \cite[pages 4--6]{hayman59}.
The special case $f(0)=0$ or $\infty$ is treated as follows:
if $f$ has a zero or pole at $0$, choose $m\in\bZ$ such that $\lim_{z\to 0} f(z)/z^m=C_m\ne
0$.
Then
\begin{align} \label{jens3}
\int_0^r \frac{n(t)}{t} dt = \frac{1}{2\pi} \int_0^{2\pi} \log\abs{f(r
  e^{i\theta})} d\theta -
\log\abs{C_m} -m\log r.
\end{align}

After this excursion to meromorphic functions  we can now prove an analogue of Theorem~\ref{tm1}.

\begin{tm} \label{th_zeros_sec}
Let $f\in \sisp ^\infty (\psi _a )\setminus \{0\}$ and $N_f$ its set of real zeros
with multiplicities $\mult_f$. Then $D^-(N_f,\mult_f) \leq 1$.
\end{tm}

The main part of the proof is an estimate of the integral in Jensen's formula.

\begin{lemma}\label{prop:boundint}
 For every $f\in V^\infty(\psi _a)$  we have
\begin{equation}\label{eq:boundint}
\sup_{r>1} \frac{1}{2\pi} \int_0^{2\pi} \log\abs{f(re^{i\theta})}\,d\theta <\infty.
\end{equation}
\end{lemma}

\begin{proof}
We divide the integral into four pieces corresponding to
$$ \theta\in I_j=\left[-\frac{\pi}{4},\frac{\pi}{4}\right]+\frac{j\pi}{2},\qquad j=0,1,2,3.$$
For $\theta \in I_0\cup I_2$, we let
$$ r e^{i\theta} = \pm\sqrt{r^2-y^2}+iy\quad\mbox{where}\quad y\in
\left[-\frac{r}{\sqrt{2}},\frac{r}{\sqrt{2}}\right].$$
By  \eqref{eq:boundf}, we have
$$ \log \abs{f(re^{i\theta})} \le \log (C\|c\|_\infty) - \log \abs{2\round{\tfrac{ay}{\pi}-\tfrac{1}{2}}}$$
and (using $d\theta = \pm dy/\sqrt{r^2-y^2}$)
$$\frac{1}{2\pi} \int_{I_0\cup I_2} \log\abs{f(r e^{i\theta})}\,d\theta
\le \frac{1}{2}\log (C\|c\|_\infty) -
\frac{1}{\pi} \int_{-\frac{r}{\sqrt{2}}}^{\frac{r}{\sqrt{2}}} \log
\abs{2\round{\tfrac{ay}{\pi}-\tfrac{1}{2}}}\,\frac{dy}{\sqrt{r^2-y^2}}.$$
Note that $\log \abs{2\round{\tfrac{ay}{\pi}-\tfrac{1}{2}}}\le 0$ and $\sqrt{r^2-y^2}\ge r/\sqrt{2}$ for all $y\in
\left[-\tfrac{r}{\sqrt{2}},\tfrac{r}{\sqrt{2}}\right]$.
Therefore,
$$ - \frac{1}{\pi} \int_{-\frac{r}{\sqrt{2}}}^{\frac{r}{\sqrt{2}}}
\log
\abs{2\round{\tfrac{ay}{\pi}-\tfrac{1}{2}}}\,\frac{dy}{\sqrt{r^2-y^2}}
\le \frac{\sqrt{2}}{\pi r}
 \int_{-\frac{r}{\sqrt{2}}}^{\frac{r}{\sqrt{2}}} \Big|  \log
   \abs{2\round{\tfrac{ay}{\pi}-\tfrac{1}{2}}}\Big| \,dy .$$
For the last integral, we use the substitution $u=\tfrac{ay}{\pi}$ and
observe that the resulting integrand is even and  periodic with period $1$. This gives for all $c<d$
\begin{eqnarray*}
 \int_c^d \abs{ \log \abs{2\round{\tfrac{ay}{\pi}-\tfrac{1}{2}}}}\,dy
&=& \frac {\pi}{a} \int_{\tfrac{ac}{\pi}}^{\tfrac{ad}{\pi}} \abs{ \log \abs{2\round{u-\tfrac{1}{2}}}}\,du \\[5pt]
&\le&
\frac {2\pi}{a} \left(\frac{ad}{\pi}-\frac{ac}{\pi}+1\right)  \int_0^{1/2}  \abs{ \log  (2u)}\,du =d-c+\frac{\pi}{a},
\end{eqnarray*}
and finally
$$\frac{1}{2\pi} \int_{I_0\cup I_2} \log|f(re^{i\theta})|\,d\theta
\le \frac{1}{2}\log (C\|c\|_\infty) +
\frac{\sqrt{2}}{\pi r}\left(\sqrt{2} r+\frac{\pi}{a}\right). $$
In the same way, for $\theta \in I_1\cup I_3$ we let
$$ r e^{i\theta} = x \pm i\sqrt{r^2-x^2}\quad\mbox{where}\quad x\in
\left[-\frac{r}{\sqrt{2}},\frac{r}{\sqrt{2}}\right],$$
and obtain from \eqref{eq:boundf}
$$\frac{1}{2\pi} \int_{I_1\cup I_3} \log\abs{f(r e^{i\theta})}\,d\theta
\le \frac{1}{2}\log (C\|c\|_\infty) -
\frac{1}{\pi} \int_{-\frac{r}{\sqrt{2}}}^{\frac{r}{\sqrt{2}}} \log
\abs{a\round{x}}\,\frac{dx}{\sqrt{r^2-x^2}}.$$
The same techniques as before give
$$ - \frac{1}{\pi} \int_{-\frac{r}{\sqrt{2}}}^{\frac{r}{\sqrt{2}}}
\log \abs{a\round{x}}\,\frac{dx}{\sqrt{r^2-x^2}}  \le
\frac{\sqrt{2}}{\pi r}
 \int_{-\frac{r}{\sqrt{2}}}^{\frac{r}{\sqrt{2}}} \Big|  \log
 \abs{a\round{x}} \Big| \,dx $$
and, for every $d>c$,
\begin{eqnarray*}
 \int_c^d \abs{ \log \abs{a\round{x}}}\,dx
&\le&  (d-c)\abs{\log (a/2)}+ \int_c^d \abs{ \log \abs{2\round{x}}}\,dx \\
&\le&  (d-c)(\abs{\log (a/2)}+ 2(d-c+1)\int_0^{1/2} \abs{ \log
      (2u)}\,du \\
&\le &  \left(d-c+1\right) \left(1+\abs{\log (a/2)}\right) .
\end{eqnarray*}
Hence, we obtain
$$\frac{1}{2\pi} \int_{I_1\cup I_3} \log|f(re^{i\theta})|\,d\theta
\le \frac{1}{2}\log (C\|c\|_\infty) +
\frac{\sqrt{2}}{\pi r}\left(\sqrt{2} r+1\right)\left(1+\abs{\log (a/2)}\right). $$
Combining both integrals, we get for $r\ge 1$
$$ \frac{1}{2\pi} \int_{0}^{2\pi} \log|f(re^{i\theta})|\,d\theta \le \log (C\|c\|_\infty )
+
\frac{\sqrt{2}}{\pi r}\left(2\sqrt{2} r+1+\frac{\pi}{a}+
(\sqrt{2} r+1)\abs{\log (a/2)}\right) ,  $$
which is bounded for $r\geq 1$. This completes the proof.
\end{proof}

\begin{proof}[Proof of Theorem \ref{th_zeros_sec}] Assume $D^-(N_f)>1$. Let $n_z(r)$ denote the
number
of zeros of $f$ in $\overline{B(0,r)}$ and
$n_p(r)$ the number of poles in that disk (both counted with multiplicities). The same
counting argument as in
the proof of Theorem \ref{tm1},
this time invoking Lemma \ref{l4},
gives
$$ n_z(r) \ge a\beta r^2$$
for some $\beta>1$ and for all $r\ge R_1$. On the other hand,
by Lemma \ref{l3}, the poles of $f$
are simple and contained in the shifted lattice
$\mathcal{L}:=\bZ+\tfrac{i\pi}{a}(\tfrac{1}{2}+\bZ)$. To find an upper
bound for $n_p(r)$, we  place  rectangles
$Q_x=x+[-\tfrac{1}{2},\tfrac{1}{2}]\times
[-\tfrac{\pi}{2a},\tfrac{\pi}{2a}]$ of area $|Q_x|=\tfrac{\pi}{a}$ and
diagonal $\sqrt{1 + \pi ^2  / a^2}$ around
each pole and observe that
$$n_p(r)=   \frac{a}{\pi} \left| \bigcup_{x\in \mathcal{L}\cap\overline{ B(0,r)}} Q_x\right|
 \le \frac{a}{\pi}~ \left|B\left(0,r+\tfrac{1}{2} \sqrt{1 + \pi ^2  /
       a^2}
   \right)\right| = a\left(r+\tfrac{1}{2} \sqrt{1 + \pi ^2  /
       a^2}
   \right)^2.$$
 As a consequence, $n(r)=n_z(r)-n_p(r)\ge (\beta-1)a r^2-cr$,
for some constant $c>0$, and
$$  \int_{R_1}^{R} \frac{n(r)}{r}\,dr \ge \frac{(\beta-1)a (R^2-R_1^2)}{2}-c(R-R_1).$$
Due to Lemma  \ref{prop:boundint},
for $R \gg 1$ this contradicts Jensen's formula \eqref{jens3}.
\end{proof}

\subsection{Transference of zero sets}
The following lemma modifies \cite[Lemma 5.1]{grrost17} to include
multiplicities and allows us to compare the density of zero sets in
different shift-invariant spaces.
\begin{lemma}\label{Rolle}
Let $f\in C^{\infty}(\bR)$ be real-valued and $\mult_f:N_f\to\bN$ be the multiplicity function
of its zeros. For $a\in \bR $ let
$g=\left(aI+\frac{d}{dx}\right)f$.
Then
\begin{align} \label{c1}
    D^-(N_{g},\mult_g)\ge D^-(N_f,\mult_f).
\end{align}
For $f \in C^{N-1}(\bR)$ the same statement holds,
replacing $m_f$ and $m_g$ by the multiplicity functions
of the zeros of height at most $N$ and $N-1$, respectively.
\end{lemma}

\begin{proof}
Let $f\in C^{\infty}(\bR)$.
Note that $aI +  \tfrac{d}{dx} =  e^{-ax} \tfrac{d}{dx} e^{ax}$.
We define  $h\in C^{\infty}(\bR)$ by $h(x)= e^{ ax}f(x)$ and note that
$N_h = N_f$ with equal
 multiplicities $m_h=\mult_f$. Furthermore,
 since
$$ g(x)= \left(aI+\frac{d}{dx}\right)f(x)= e^{- ax}h'(x) ,
$$
we conclude that $N_g=N_{h'}$, again  with equal multiplicities $\mult_g=m_{h'}$.
It remains to show that $D^-(N_{h'},m_{h'})\ge D^-(N_h,m_h)$.

Let $x\in\bR$, $R>0$, and $F \subseteq N_{h}\cap[x-R,x+R]$ a finite subset.
All zeros $x$ of $h$ with multiplicity $m_h(x)>1$ are zeros of $h'$ with multiplicity
$m_{h'}(x)=m_h(x)-1$.
Since the zeros of $h$ and $h'$ are interlaced by Rolle's theorem, we
obtain
additional zeros $\tilde F\subset [x-R,x+R]\setminus F$ of $h'$ with
cardinality $\# \tilde F =  \#F-1$. Combining both types of zeros of
$h'$ gives
$$
     \sum_{x\in F\cup \tilde F} m_{h'}(x) \ge  \left(\sum_{x\in F} m_{h}(x)\right)-1.
$$
Since this holds for every finite subset $F \subseteq
N_{h}\cap[x-R,x+R]$,
it follows that
$D^-(N_{h'},m_{h'})\geq  D^-(N_h,m_h)$, and
$$
D^-(N_g,\mult_g) =  D^-(N_{h'},m_{h'})\ge D^-(N_h,m_h) = D^-(N_f,\mult_f) \, ,
$$
as claimed. Finally, for $f \in C^{N-1}(\bR)$,
$g \in C^{N-2}(\bR)$, and
the same argument applies to the multiplicity functions of zeros of height $N$ and $N-1$.
\end{proof}

Although generically one would expect  equality in \eqref{c1}, the
density of the zero set may actually jump.   Let   $f(x) =
\sum_{k\in \Zst} e^{-\pi (x-k)^2} \in V^\infty  (\phi _\pi )$,  and $h=f'
\in V^\infty (\phi '_\pi )$. Then $f$ is a non-constant, strictly
positive,  periodic,
real-valued function, and we have $N_f = \emptyset$ and $D^-(N_f) = 0$. Since $f$ assumes two
extremal values in $[0,1)$, we have $D^-(N_h)=2$, in fact, $N_h =
\tfrac{1}{2} \Zst $.  This example explains why the methods of this
paper cannot be applied directly to sampling in shift-invariant spaces
generated by Hermite functions. Indeed, Theorem~\ref{th_samp_der_tp} does not have a
direct analog for $\sisp (h_n)$ with the $n$-th Hermite function $h_n,
n>0$.

\subsection{Totally positive functions of Gaussian type}
We next study shift-invariant spaces generated by a totally positive
function of Gaussian type and their density of zeros.
\begin{tm}
\label{th_zeros_tp}
Let $g$ be a totally positive function of Gaussian type.
Let $f\in \sisp ^\infty (g)\setminus\{0\}$ be real-valued and $(N_f, \mult_f)$ its set of real zeros
counted with multiplicities. Then $D^-(N_f, \mult_f) \leq 1$.

In particular, if $D^-(\Lambda )>1$, then  $\Lambda $ is a uniqueness
set for $\sisp ^\infty (g)$.
\end{tm}

\begin{proof}
  The proof is an adaption of the argument in \cite{grrost17} using
  multiplicities. Recall that $g$ is real-valued and has stable
  integer shifts. Let $c\in\ell^\infty(\bZ)$ and assume that $f=\ctilg \in \sisp^\infty(g)$
vanishes on $N_f\subset \bR$ with $D^-(N_f,\mult_f ) >1$. We want to show that $f \equiv 0$.
Note that $f\in C^\infty(\bR)$. Since $g$ is real-valued, we may
assume without loss of generality that  $f$ is also real-valued (by  replacing $c_k$ by
$\Re(c_k)$ or $\Im(c_k)$ if necessary).

Using \eqref{eq:tpaa}, write
\begin{equation}\label{eq:ch1a}
  \hat g(\xi)= \prod_{j=1}^n (1+2\pi i\delta_j\xi)^{-1} \, \hat\phi(\xi),\qquad
\delta_1,\ldots,\delta_n\in\bR\setminus\{0\}, \ c >0,
\end{equation}
where $\hat{\phi}(\xi)=e^{-c\xi^2}$. In other words, $\phi  =
\prod_{j=1}^n \left(I+\delta_j \tfrac{d}{dx}\right) g $ is a Gaussian.
Since  $\phi$, $g$, and their derivatives decay exponentially, we may
interchange summation and differentiation in $f$, and obtain that
$$
h = \prod_{j=1}^n \left(I+\delta_j \frac{d}{dx}\right) f \in \sisp^\infty(\phi).
$$
The repeated use of Lemma \ref{Rolle} implies that $D^-(N_h,\mult_h) \geq D^-(N_{f},\mult_f)>1$. Hence, by
Theorem \ref{tm1}, $h=\sum_k c_k \phi(\cdot-k) \equiv 0$. Hence $c_k \equiv
0$ and $f \equiv 0$, as claimed.
\end{proof}

\subsection{Bandlimited functions}
For a simple comparison of the results in Theorems~\ref{tm1}, \ref{th_zeros_sec}, and \ref{th_zeros_tp},  we mention the following result for bandlimited functions.
\begin{tm}
\label{th_uniqueness_pw}
Let $f \in \mathrm{PW}^\infty \setminus \{0\}$. Then $D^-(N_{f},\mult_f) \leq 1$.
\end{tm}
\begin{proof}
The result follows from the Paley-Wiener characterization of bandlimited functions as
restrictions of entire functions of exponential type, and Jensen's formula. Beurling's
proof \cite{be66,be89} applies almost verbatim.
\end{proof}

\section{Proof of the sampling theorems}
\label{sec_proofs}

The proofs of our main theorems are now short and follow from the
combination of the characterization of sampling sets without
inequalities (Theorem~\ref{th_wl_tp}) and the new insights about the density
of zero sets in shift-invariant spaces (Section~4).

\begin{proof}[Proof of Theorem \ref{th_samp_der_tp}]
The necessity of the density conditions is stated in Proposition \ref{prop_nec}.
For the sufficiency, we apply the characterization of Theorem \ref{th_wl_tp}. Suppose that
$D^{-}(\Lambda, \ml)>1$, and let $(\Gamma, \mg) \in \WZ(\Lambda, \ml)$. By Lemma
\ref{lemma_sep}, $D^{-}(\Gamma, \mg)>1$. Hence, by Theorem \ref{th_zeros_tp}, $(\Gamma, \mg)$
is a uniqueness set for $\sisp ^\infty (g)$. Therefore, the criterion in Theorem
\ref{th_wl_tp} is satisfied, and we conclude that $(\Lambda, \ml)$ is a sampling set
for $\sisp ^2 (g)$.
\end{proof}

\begin{proof}[Proof of Theorem \ref{th_samp_der_sec}]
The proof is the same as for  Theorem \ref{th_samp_der_tp}; this time
we  resort  to
Theorem \ref{th_zeros_sec} (instead of Theorem~\ref{th_zeros_tp}).
\end{proof}

\begin{proof}[Proof of Theorem \ref{th_samp_der_pw}]
The first part of the proof (treating $\mathrm{PW}^\infty $) is similar to the
one of Theorem \ref{th_samp_der_tp}. If  $(\Lambda, \ml)$ is a separated set with finite height and
density $D^{-}(\Lambda, \ml) > 1$, then Theorem \ref{th_wl_pw} (combined with Lemma~\ref{lemma_sep} and Theorem~\ref{th_uniqueness_pw}) implies
 that $(\Lambda, \ml)$ is a sampling set
for $\mathrm{PW}^\infty$.   As a second step, we use Proposition \ref{prop_pw} to extend the
conclusion to $\mathrm{PW}^2$. More precisely, if $D^{-}(\Lambda, \ml) > 1$, we select $\alpha
<1$ such that $\alpha D^{-}(\Lambda, \ml) = D^{-}(\alpha^{-1} \Lambda, \ml) > 1$. We conclude
that
$(\alpha^{-1} \Lambda, \ml)$ is a sampling set for $\mathrm{PW}^\infty$, and therefore, by
Proposition \ref{prop_pw}, $(\Lambda, \ml)$ is a sampling set for $\mathrm{PW}^2$.
\end{proof}

\section{Consequences for Gabor frames}

The Hermite-sampling results of Theorems~\ref{th_samp_der_tp} and~\ref{th_samp_der_sec} can be applied in order to obtain sharp density results for multi-window Gabor frames. This extends our previous work in  \cite{grrost17} and was,
in fact, one of our original  motivations for the present work. We obtain new families of multi-window Gabor frames with
optimal conditions for semi-regular sets of time-frequency shifts.

\subsection{Multi-window Gabor frames}
Let  $\pi(x,w)g(t)= g(t-x) e^{2\pi i w t}$ denote the time-frequency
shift of $g$ by $(x,w) \in \Rst \times \Rst  $.
For given windows $g^1, \ldots, g^N \in L^2(\Rst)$ and sets
$\Delta^1, \ldots, \Delta^N \subseteq \Rst^2$
the associated  multi-window Gabor system   is
\begin{align}
\mathcal{G}(g^1, \ldots, g^N, \Delta^1, \ldots \Delta^N)
=\sett{ \pi(x,w)g^j  \, : \,(x,w) \in \Delta^j, j=1,\ldots,N}\, .
\end{align}

It will be convenient to use the notation $G=(g^1,\ldots,g^N)$,
$\vec\Delta=(\Delta^1,\ldots,\Delta^N)$ and
$\mathcal{G}(G, \vec\Delta)$. When all the sets $\Delta^j$ are equal, we just write
$\mathcal{G}(G, \Delta)$.

\subsection{Connection between sampling and Gabor frames}
For semi-regular sets $\vec\Delta$, the Gabor frame property can be
related to a sampling problem as follows.
\begin{tm}
\label{tm_gab_con}
Assume that  $G=(g^1, \ldots, g^N) \in (W_0(\Rst))^N$ has stable integer shifts
and that the sets $\Lambda^1, \ldots, \Lambda^N \subseteq \Rst$ are separated.
Let
$\vec\Delta=(\Delta^1,\ldots,\Delta^N)$ be given by
$\Delta^j := (-\Lambda^j)\times\Zst$.

Then $\mathcal{G}(G,\vec\Delta)$ is a frame for $L^2(\Rst)$ if
and only if $\vec\Lambda+(x,\ldots,x)$ is a sampling set for $\sisp^2(G)$ for all $x \in \Rst$.
\end{tm}
Theorem \ref{tm_gab_con} is a vector-valued extension of
\cite[Theorem 2.3]{grrost17} (equivalence of conditions (a) and (b)),
and we therefore  omit its proof.

\subsection{Characterization of  multi-window Gabor frames with
  totally positive windows}
\begin{tm}
\label{th_gab_1}
Let $g$ be a totally positive function of Gaussian-type or the
hyperbolic secant and  let $\Lambda \subseteq \Rst$ be a separated set.
Let $\{ p_1, \dots p_N\} $ be a basis of the space of polynomials of
degree less than $N$. Set   $g^j = p_j\big( \tfrac{d}{dx}\big) g$ and  $G=(g^1,\ldots,g^N)$.

Then $\mathcal{G}(G, (-\Lambda) \times \Zst)$ is a frame for $L^2(\Rst)$
if and only if $D^-(\Lambda) > 1/N$.
\end{tm}
\begin{proof}
The necessity of the condition $D^-(\Lambda) >  1/N$ for multi-window Gabor frames over a rectangular lattice is contained
in~\cite[Thm.~12.2.11]{zzgab}. It also   follows from
general results, see, e.g.,  \cite{grorro15}.

For the sufficiency,
assume that $D^-(\Lambda) > 1/N$, and
let $\vec\Lambda := (\Lambda,\ldots,\Lambda)$. By Theorem \ref{tm_gab_con},
it suffices to show that for all $x\in \Rst$, $\vec\Lambda+\vec x$ is a sampling set
for the vector-valued shift-invariant space $V^2(G)$,
where $\vec x := (x,\ldots,x)$. To verify this condition,
we apply Theorem \ref{th_wl_tp}.

Let $\vec\Gamma \in \WZ (\vec\Lambda+\vec x)$.  This set is necessarily of the form
$\vec\Gamma=(\Gamma, \ldots, \Gamma)$, for some $\Gamma \in \WZ  (\Lambda+x)$,
and, by Lemma \ref{lemma_sep}, $D^{-}(\Gamma) > 1/N$.
Assume that $F \in \sisp^\infty(G)$ vanishes on $\vec\Gamma$. We need
to show
that $F \equiv 0$. Explicitly $F$ is given by an expansion  $F =
\sum_{k \in \Zst} c_k G(\cdot-k)$ with $c\in \ell ^\infty (\bZ
)$.

We now relate the sampling problem for vector-valued functions to a
sampling problem with derivatives. To do this, we set $P=(p_1, \dots ,
p_N)$ and $Q= (1, x, \dots , x^{N-1})$. By assumption on $P$, there is
an invertible $N\times N$-matrix $B$, such that $BP=Q$, i.e., $x^{j-1} = \sum
_{k=1}^N b_{jk} p_k(x)$ for $j=1, \dots, N$ and thus
\begin{equation}
  \label{eq:h9}
\sum   _{k=1}^N b_{jk} g^k  = \sum _{k=1}^N b_{jk}  p_k\big(
\tfrac{d}{dx}\big) g =    g^{(j-1)} \, .
\end{equation}
Consequently,
after taking linear combinations of translates we obtain
$$
BF (x) _j  = \sum _{l\in \bZ } c_l BG(x-l)_j = \sum _{l\in \bZ } c_l
g^{(j-1)}(x-l) = f^{(j-1)}(x) \, ,
$$
where $f=\sum_l c_l g(\cdot-l)\in \sisp^\infty(g)$ is the first component of $BF$.
If $F$ vanishes on $\vec\Gamma $, then also  $f^{(j-1)} $
vanishes on $\Gamma $ for $j=1, \dots , N$.
Hence, $f$ vanishes on $\Gamma$ with multiplicity $N$ and $D^-(N_f,
m_f) \geq N D^-(\Gamma ) >1$.
By Theorem \ref{th_zeros_tp} or \ref{th_zeros_sec}, this implies that $f \equiv 0$.
Hence, $c_k \equiv 0$ and $F \equiv 0$, as desired.
\end{proof}

We single out two special cases of Theorem~\ref{th_gab_1}.

\begin{cor}
Let $g$ be a totally positive function of Gaussian-type or the
hyperbolic secant and  let $\Lambda \subseteq \Rst$ be a separated set.
Let $a_1, a_2, \dots , a_{N-1} \in \bR $,   $g^1=g$, and set
\begin{align}
 g^j := \prod_{k=1}^{j-1} \left(a_{k} I+ \tfrac{d}{dx}\right) g,
  \qquad j=2,\ldots,N ,
 \end{align}
and  $G=(g^1,\ldots,g^N)$.
Then $\mathcal{G}(G, \Lambda \times \Zst)$ is a frame for $L^2(\Rst)$
if and only if $D^-(\Lambda) > 1/N$.
\end{cor}

For the second corollary we use  the basis of Hermite functions $\{h_k: k \geq
0\}$ which is defined by
$$
  h_k(x)= \gamma_k e^{\pi x^2} ~\frac{d^{k}}{dx^{k}} e^{-2\pi x^2} = (-1)^k\gamma_k e^{-\pi x^2} ~H_k(x),
$$
with the Hermite polynomials $H_k$ of degree $k$ and some  normalizing constant $\gamma_k>0$.

\begin{cor}
\label{coro_hermite}
Let $\Lambda \subseteq \Rst$ be a separated set and  $b>0$.
Then $\mathcal{G}(h_0, \ldots, h_{N-1}, \Lambda \times b \Zst)$ is a frame of $L^2(\Rst)$
if and only if $D^{-}(\Lambda) > b/N$.
\end{cor}
\begin{proof}
We use the fact that $\mathcal{G}(h_0, \ldots, h_{N-1}, \Lambda \times b \Zst)$ is a frame
if and only if
\[\mathcal{G}( h_0(b\inv \cdot), \ldots,  h_{N-1}(b\inv \cdot), b\Lambda \times \Zst),\]
is.
Because the Hermite
polynomials $H_k, k=0, \dots , N-1$, form a basis for the polynomials
of degree $<N$, the span of $h_k$, $0\le k\le N-1$, is the same as the span of all derivatives $\tfrac{d^{j}}{dx^{j}}e^{-\pi x^2}$, $0\le j\le N-1$. The result is a consequence of Theorem \ref{th_gab_1}.
\end{proof}

Corollary \ref{coro_hermite} actually follows from a sampling result
of Brekke and Seip in Fock space
~\cite{bese93}.  It can also  be
reformulated for spaces of  polyanalytic functions. For this
connection   see
 \cite{ab10}.

\section{Postponed proofs}
\label{sec_post}
\subsection{Proof of Proposition \ref{prop_wstar}}
For sets without multiplicities, i.e., $m_\Lambda
\equiv 1$, the proposition is classical.

Let $(\Lambda, \ml)$ be a
separated set with multiplicity   with finite height,
let $(\Gamma, \mg)$ be a set with multiplicity, and $\{k_n: n \geq 1\}
\subseteq \Zst$. Recall that $\Lambda ^j = \{ \lambda \in \Lambda :
m(\lambda ) \geq j\}$.

Suppose first that $\Lambda^j-k_n \weakconv \Gamma^j$, as $n \longrightarrow \infty$ for all
$j=1,\ldots,N$. Then, by the  case without multiplicity,
$\sum_{\lambda \in \Lambda^j} \delta_{\lambda-k_n} \longrightarrow
\sum_{\gamma \in \Gamma^j} \delta_\gamma$,
in the $\sigma(C^*_c,C_c)$ topology. Since $(\Lambda, \ml)$ has finite height,
the claim follows by summing over  $j$.

Conversely, assume that
$\mu_n:= \sum_{\lambda \in \Lambda} \ml(\lambda) \delta_{\lambda-k_n} \longrightarrow
\mu:= \sum_{\gamma \in \Gamma} \mg(\gamma) \delta_\gamma$,
in the $\sigma(C^*_c,C_c)$ topology. As discussed in
\cite[Lemmas 4.3, 4.4]{grorro15}, it follows that
\[
\Lambda^1-k_n=\Lambda-k_n = \supp(\mu_n) \weakconv \supp(\mu) = \Gamma = \Gamma^1.
\]
(Here it is crucial
that the multiplicities  $\ml(\lambda)$, $\mg(\gamma)$ are integers.)
It remains to show that $\Lambda^j-k_n \weakconv \Gamma^j$ for $j>1$.
Since $\Lambda^1-k_n \weakconv \Gamma^1$, the case without
multiplicity  implies
that $\sum_{\lambda \in \Lambda} \delta_{\lambda-k_n} \longrightarrow
\sum_{\gamma \in \Gamma} \delta_\gamma$. Therefore,
\begin{equation}
\label{eq_abcd}
\sum_{\lambda \in \Lambda} (\ml(\lambda)-1) \delta_{\lambda-k_n} \longrightarrow
\sum_{\gamma \in \Gamma} (\mg(\gamma)-1) \delta_\gamma.
\end{equation}
Since $(\Lambda,\ml)$ has finite height, we can proceed by induction.
Indeed, we consider the sets $\Lambda_0:=\Lambda^2$ and $\Gamma_0:=\Gamma^2$, with
multiplicites
$\mult_{\Lambda_0} := \ml-1$ and $\mult_{\Gamma_0}:=\mg-1$, and note that
$\Lambda^j_0=\Lambda^{j+1}$ and $\Gamma^j_0=\Gamma^{j+1}$. $\qed$

\subsection{Sketch of a proof of Theorem \ref{th_wl_vec}}
Let $I:= \{(\lambda,j) \in \Rst^2: \lambda \in \Lambda^j, j=1,\ldots,N\}$ and consider the
matrix $A \in \bC^{I \times \Zst}$, given by
\begin{align*}
A_{(\lambda,j), k} := G^j(\lambda-k).
\end{align*}
Then $\vec\Lambda$ is a sampling set for $\sisp(G)$ if and only if
$A:\ell^p(\Zst) \to \ell^p(I)$ is bounded below. The independence of $p$ of this property
for the range $p \in [1,+\infty]$ follows from (a slight extension of)
Sj\"ostrand's Wiener-type lemma \cite{sj95}. The formulation in
\cite[Proposition A.1]{grrost17} is applicable directly.  Specifically, \cite[Proposition A.1]{grrost17}
concerns a matrix indexed by two \emph{relatively separated} subsets
of the Euclidean space (where a relatively separated set
is just  a finite union of separated sets).
In our case, $I$ is a relatively separated subset of $\Rst^2$,
while $\Zst$ can be embedded into $\Rst^2$ as $\Zst\times\{0\}$. This accounts for the
equivalences $(a) \Leftrightarrow (b)$. The other implications follow,
with very minor modifications, as in the proof of
\cite[Theorem 3.1]{grrost17}. See also \cite[Section 4]{grorro15} for some relevant technical
tools. $\qed$.

\subsection{Sketch of a proof of Proposition \ref{prop_nec}}
The proposition follows from the theory of density of frames. The Paley-Wiener case is
explicitly treated in
\cite{grra96} following the technique of Ramanathan and Steger \cite{rast95}.
For shift-invariant spaces with generators in $g \in W^N_0(\Rst)$, we can use the abstract
density
results for frames from \cite{bchl06} as follows.

Suppose that $(\Lambda,\ml)$ is a sampling set for $\sisp^2(g)$.
By assumption, the Bessel map, $\ell^2(\Zst) \ni c \to \sum_k c_k
g(\cdot-k) \in \sisp^2(g)$, is an isomorphism. The sampling inequality
\eqref{eq:lpstable} with $p=2$ means that the set $\mathcal{F}$ formed by the
 sequences
\begin{align*}
\varphi_{\lambda,j} :=
\left(g^{(j)}(\lambda-k)\right)_{k \in \Zst}, \qquad \lambda \in \Lambda, j=0, \ldots,
\ml(\lambda)-1,
\end{align*}
is a frame for $\ell^2(\Zst)$. We consider the index set
$I:=\{(\lambda,j) \in \Rst^2: \lambda \in \Lambda, j=0, \ldots,
\ml(\lambda)-1\}$ and a map $\alpha:I \to \Zst$ such that
$\alpha(\lambda,j)=l$, with $\abs{l-\lambda} \leq 1/2$. Second, we let
$\Phi(x) := \sum_{j=0}^{N-1} \max_{y: \abs{y-x} \leq 1} \abs{g^{(j)}(y)}$.
Since $g \in W_0^N(\Rst)$, it follows that $\Phi \in W_0(\Rst)$, and we have the estimate
\begin{align*}
\abs{\varphi_{\lambda,j}(k)} =\frac{}{} \abs{g^{(j)}(\lambda-k)} \leq \Phi(\alpha(\lambda)-k),
\end{align*}
which, in the terminology of \cite{bchl06} means that $\mathcal{F}$ is $\ell^1$-localized with
respect to the canonical basis of $\ell^2(\Zst)$. The comparison theorem
\cite[Thm.~3]{bchl06} yields the estimate $D^{-}(I,\alpha) \geq 1$
in terms of the density
\begin{align*}
D^{-}(I,\alpha)=\liminf_{n \longrightarrow \infty} \inf_{k \in \Zst}
\frac{\# \alpha^{-1}([k-n,k+n])}{\# [k-n,k+n]} \, .
\end{align*}
Clearly $D^-(I,\alpha )$  coincides with $D^{-}(\Lambda, \ml)$.

Alternative arguments can be given by checking the general conditions
in  \cite{fghkr} or \cite{ro11}.
$\qed$.

\end{document}